\documentclass[a4paper,12pt,final]{amsart}
\usepackage{version}
\usepackage[notcite,notref]{showkeys}
\usepackage{amssymb}
\usepackage{amsfonts}
\usepackage{graphicx}
\usepackage{tikz}
\usetikzlibrary{arrows}
\usepackage{color}
\usepackage{epstopdf}
\usepackage[english]{babel}
\usepackage{float}
\usepackage{cite}

\usepackage{geometry}
\newtheorem{theorem}{Theorem}[section]

\newtheorem{lemma}[theorem]{Lemma}

\newtheorem{proposition}[theorem]{Proposition}

\theoremstyle{definition}

\theoremstyle{remark}
\newtheorem{remark}[theorem]{Remark}

\numberwithin{equation}{section}

\renewcommand\bigskip{\medskip}

\def\to{\rightarrow}

\def\N{\mathbb N}
\def\M{\mathcal M}

\def\R{\mathbb R}

\def\i{{\bf i}}
\def\vep{\varepsilon}

\DeclareMathOperator{\dimH}{dim_H}
\DeclareMathOperator{\dimp}{dim_P}
\DeclareMathOperator{\dimb}{dim_B}

\newcommand{\bi}{{\bf i}}
\newcommand{\bj}{{\bf j}}
\newcommand{\ba}{{\bf a}}
\newcommand{\bN}{{\bf N}}

\begin{document}

\title[Random affine code tree fractals]
{Random affine code tree fractals: Hausdorff and affinity dimensions and 
pressure}

\author[E. J\"arvenp\"a\"a]{Esa J\"arvenp\"a\"a}
\address{Department of Mathematical Sciences, P.O. Box 3000,
  90014 University of Oulu, Finland}
\email{esa.jarvenpaa@oulu.fi}

\author[M. J\"arvenp\"a\"a]{Maarit J\"arvenp\"a\"a}
\address{Department of Mathematical Sciences, P.O. Box 3000,
  90014 University of Oulu, Finland}
\email{maarit.jarvenpaa@oulu.fi}

\author[M. Wu]{Meng Wu}
\address{Department of Mathematical Sciences, P.O. Box 3000, 90014 
University of Oulu, Finland}
\email{meng.wu@oulu.fi}

\author[W. Wu]{Wen Wu}
\address{Department of Mathematical Sciences, P.O. Box 3000, 90014 
University of Oulu, Finland\newline 
\indent School of Mathematics and Statistics, Hubei 
University, Wuhan 430062, P.R. China}
\email{wen.wu@oulu.fi; hust.wuwen@gmail.com}

\thanks{We acknowledge the support of Academy of Finland, the Centre of
Excellence in Analysis and Dynamics Research. Wen Wu was also supported by 
NSFC (Grant Nos. 11401188)}


\begin{abstract}
We prove that for random affine code tree fractals 
the affinity dimension is almost surely equal to the unique zero of the 
pressure function. As a consequence, we show that the Hausdorff, packing
and box counting dimensions of 
such systems are equal to the zero of the pressure. In particular, we do not
presume the validity of the Falconer-Sloan condition or any other additional 
assumptions which have been essential in all the previously known results.
\end{abstract}

\maketitle

\section{Introduction}\label{intro}

The investigation of dimensional properties of self-affine sets dates 
back to the pioneering works of Bedford \cite{Be} and McMullen \cite{Mc}.
Besides considering specific self-affine sets as in \cite{Be,Mc},
a natural approach is to seek generic dimension formulae. In \cite{Fa}, 
Falconer obtained a dimension formula for generic self-affine sets in terms of 
the pressure function under the assumption that the norms of 
the linear parts of generating functions are less than 1/3. Later, Solomyak
\cite{So} observed that 1/3 can be replaced by 1/2 which, in turn, is the best 
possible bound due to an example given by Przytycki and Urba\'nski \cite{PU}.
Since the seminal works in \cite{Be,Fa,Mc}, there has been 
great interest in various problems related to dimensions of self-affine sets. 
For recent contributions in this field, see for example 
\cite{BF,D,Fa2,FM,FM2,FM3,FW,FJR,FJS,Fr,Fr2,FrSh,JPS,JR,Ro,S} and the 
references therein. 

We address the problem of studying a general class of random affine code tree
fractals introduced in \cite{JJKKSS}. Typically, code tree fractals 
are locally random but globally nearly homogeneous mimicking deterministic 
systems. As explained in \cite{JJKKSS,JJLS}, our setup includes several random 
models, for example, homogeneous graph directed systems \cite{MU03} and 
$V$-variable fractals \cite{BHS}.

In \cite{Fa}, Falconer proved that the Hausdorff 
dimension of the attractor of an affine iterated function system is equal to
the zero of the pressure for almost all translation vectors, and well-known 
examples show that this is genuinely an almost all type of result.
Using Carath\'eodory's construction with weights determined by the 
singular value function, he introduced a parametrised family 
of net measures which behaves like the Hausdorff measure in the sense that 
there exists a unique parameter value where the measure drops from infinity 
to zero. In what follows, this unique value is called the affinity dimension. 
According to \cite{Fa}, the affinity dimension equals the Hausdorff dimension 
for almost all translation vectors and, moreover, the affinity dimension
is equal to the unique zero of the pressure, which always exists for 
self-affine sets. 

As far as the extendability of Falconer's result for code tree fractals
is concerned, in \cite{JJKKSS} it is proved that, for all code tree fractals, 
the affinity dimension equals the Hausdorff dimension for almost all 
translation vectors (see also Theorem~\ref{theorem1}). However, the Hausdorff 
or affinity dimension need not to be equal to the zero of the pressure.
In particular, the pressure does not necessarily exist. 
To avoid this kind of irregular behaviour, a general class of random code 
tree fractals with a neck structure was introduced in \cite{JJKKSS}. In this 
class, the pressure exists and has a unique zero almost surely with 
respect to any shift invariant ergodic probability measure for which 
the expectation of the length of the first neck is finite (see \cite{JJKKSS} 
or Theorem~\ref{pexists}). The problem whether the zero of the pressure is 
equal to the Hausdorff dimension was addressed in \cite{JJKKSS}. It turned out
that in the plane under several additional assumptions this is 
indeed the case (see \cite[Theorem 5.1]{JJKKSS}).

In many questions involving the pressure of affine systems, the fact that the 
singular value function is submultiplicative but not multiplicative causes
problems (see, for example, \cite{FM,FS,Fe,FeSh}).
In our context, the Falconer-Sloan condition, introduced in \cite{FS}, is 
useful for the 
purpose of overcoming these problems. Indeed, in \cite{JJLS} it is verified 
that in the $d$-dimensional case
the Hausdorff dimension of a typical code tree fractal equals the zero of 
the pressure under a weak probabilistic version of the Falconer-Sloan 
condition (\cite[Theorem 3.2]{JJLS}). As pointed out in \cite{JJLS}, 
the Falconer-Sloan condition, or another corresponding assumption, 
is necessary for the method of proof in \cite{JJLS}. The question whether
the result could be true without additional assumptions remained open.
In this paper, we answer this question by proving that this is indeed the case.
Our main theorem states that for typical code tree fractals the Hausdorff, 
packing and box counting dimensions are equal to the zero of the pressure 
(see Theorem~\ref{main theorem}). 

Our methods are completely different from those of \cite{JJKKSS,JJLS}. We 
introduce a new concept of a neck net measure (see 
Section~\ref{main}) resembling the net measure that is utilised
in the definition of the affinity dimension - the essential difference 
being that
the neck net measure takes into account the neck structure. It turns out
that the affinity dimension is almost surely equal to the neck affinity
dimension (Proposition~\ref{proposition equality of alpha and alpha tilde})
which, in turn, equals the zero of the pressure 
(Theorem~\ref{main theorem}). The main part of the proof is to show that the 
neck affinity dimension is not smaller than the zero of the pressure. The proof
of this fact is based on a careful decomposition of trees, appearing in the 
definitions of the pressure and the neck affinity dimension, into suitable 
subtrees.

The paper is organised as follows. In Section~\ref{notation}, we recall the
notation from \cite{JJKKSS,JJLS} and present results needed for 
proving our main theorem. In Section~\ref{main}, we state and prove the main 
result.

\section{Notation and preliminaries}\label{notation}

In this section, we summon notation and preliminaries from \cite{JJKKSS}.
(Note that, for simplicity, our notation is slightly different from that of
\cite{JJKKSS}.) Let $\Lambda$ be a topological space. Suppose that 
$\mathcal F
  =\{F^\lambda=\{T_i^\lambda+a_i^\lambda\}_{i=1}^{M_\lambda}|\lambda\in\Lambda\}$ 
is a family of iterated function systems on $\mathbb R^d$ consisting of 
affine maps. Here $T_i^\lambda:\R^d\to\R^d$ is a non-singular linear mapping and 
$a_i^\lambda\in\mathbb R^d$. For brevity, we write 
$f_i^\lambda=T_i^\lambda+a_i^\lambda$ for all $i=1,\dots,M_\lambda$ and
$\lambda\in\Lambda$. We assume that 
\begin{equation}\label{equation introduction 0}
\sup_{\lambda\in\Lambda,i=1,\dots,M_\lambda}|a_i^\lambda|<\infty
\end{equation}
and, moreover, there exist $M>0$ and 
$\underline{\sigma},\overline{\sigma}\in (0,1)$ such that
\begin{equation}\label{equation introduction 1}
M=\sup_{\lambda\in \Lambda} M_\lambda<\infty
\end{equation}
and 
\begin{equation}\label{equation introduction 2}
 \underline{\sigma}\leq \sigma_d(T^\lambda_i)\leq \cdots \leq 
\sigma_1(T^\lambda_i)=\Vert T^\lambda_i\Vert\leq \overline{\sigma}
\text{ for all }\lambda\in 
\Lambda\text{ and } i=1,...,M_\lambda
\end{equation}
where $\sigma_j(T)$ is the $j$-th singular value of 
a non-singular linear mapping $T:\R^d\to \R^d$. 
Note that $F^\lambda$ may be naturally identified with an element of 
$\R^{(d^2+d)M_\lambda}$ and, thus, $\mathcal F\subset\bigcup_{i=1}^M\R^{(d^2+d)i}$, 
where the disjoint union is endowed with the natural topology. We suppose that 
the map $\lambda\mapsto F^\lambda$ is Borel measurable.
For $s\geq 0$, we denote by $\Phi^s(T)$ the multiplicative 
{\it singular value function} of a linear map $T:\R^d\to\R^d$, that is, 
\[
\Phi^s(T)
 =\begin{cases}\sigma_1(T)\sigma_2(T)\cdots\sigma_{m-1}(T)\sigma_m(T)^{s-m+1},&
                \text{if } 0\le s\le d\\
   \sigma_1(T)\sigma_2(T)\cdots\sigma_{d-1}(T)\sigma_d(T)^{s-d+1},&\text{if }s>d
  \end{cases}
\]
where $m$ is the integer such that $m-1\le s<m$. 

Let $I=\{1,\dots,M\}$ and $I^0=\{\emptyset\}$. For all $k\in\N$, the length of 
a word $\tau\in I^k$ is $|\tau|=k$. We associate to a function 
$\tilde\omega\colon\bigcup_{k=0}^\infty I^k\to\Lambda$ a tree rooted at 
$\emptyset$ in a natural manner: Let 
$\Sigma^{\tilde\omega}_{*}\subset\bigcup_{k=0}^\infty I^k$ be the unique set such that
\begin{itemize}
\item $\emptyset\in\Sigma^{\tilde\omega}_*$,
\item if $i_1\cdots i_k\in\Sigma^{\tilde\omega}_*$ and
   $\tilde\omega(i_1\cdots i_k)=\lambda$, then 
   $i_1\cdots i_kl\in\Sigma^{\tilde\omega}_*$ if and only if $l\leq M_\lambda$, 
\item if $i_1\cdots i_k\notin\Sigma^{\tilde\omega}_*$, then for all $l=1,\dots,M$,
   we have $i_1\cdots i_kl\notin\Sigma^{\tilde\omega}_*$.
\end{itemize}
The restriction of $\tilde\omega$ to $\Sigma^{\tilde\omega}_*$ is called 
{\it a code tree}.
In a code tree, we identify the vertex $i_1\cdots i_k$ with the function 
system $F^{\tilde\omega(i_1\cdots i_k)}$ and, moreover, the edge connecting 
$i_1\cdots i_k$ to $i_1\cdots i_kl$ with the map $f_l^{\tilde\omega(i_1\cdots i_k)}$.
Let $\widetilde\Omega$ be the set of all code trees. {\it A sub code tree} of 
a code tree $\tilde\omega$ is the restriction of $\tilde\omega$ 
to a subset of $\Sigma_*^{\tilde\omega}$ which is rooted at some vertex 
$i_1\cdots i_k\in\Sigma_*^{\tilde\omega}$ and contains all descendants of 
$i_1\cdots i_k$ belonging to $\Sigma_*^{\tilde\omega}$. We endow 
$\widetilde\Omega$ with the topology generated by the sets 
\[
\{\tilde\omega\in\widetilde\Omega\mid\Sigma_*^{\tilde\omega}\cap\bigcup_{j=0}^k 
 I^j=J\text{ and }\tilde\omega(\bi)\in U_\bi\text{ for all }\bi\in J\},
\]
where $k\in\mathbb N$, $U_\bi\subset\Lambda$ is open for all $\bi\in J$ and 
$J\subset\bigcup_{j=0}^k I^j$ is a tree rooted at $\emptyset$ and having all 
leaves in $I^k$. 

Equip $I^{\mathbb N}$ with the product topology, and define for all 
$\tilde\omega\in\widetilde\Omega$ 
\[
\Sigma^{\tilde\omega}=\{\bi=i_1i_2\cdots\in I^{\mathbb N}\mid
  i_1\cdots i_n\in\Sigma^{\tilde\omega}_*\text{ for all }n\in\mathbb N\}.
\]
Then $\Sigma^{\tilde\omega}$ is compact. For all $k\in\mathbb N$ and 
$\bi\in\Sigma^{\tilde\omega}\cup\bigcup_{j=k}^\infty I^j$, we
denote by $\bi_k=i_1\cdots i_k$ the initial word of $\bi$ with length $k$ and
use following type of natural abbreviations for compositions:
\[
f^{\tilde\omega}_{\bi_k}=f_{i_1}^{\tilde\omega(\emptyset)}\circ f_{i_2}^{\tilde\omega(i_1)}
  \circ\dotsb\circ f_{i_k}^{\tilde\omega(i_1\cdots i_{k-1})}\text{ and }
T_{\bi_k}^{\tilde\omega}=T_{i_1}^{\tilde\omega(\emptyset)}T_{i_2}^{\tilde\omega(i_1)}\cdots
  T_{i_k}^{\tilde\omega(i_1\cdots i_{k-1})}.
\]
Note that the maps $\tilde\omega\mapsto f^{\tilde\omega}_{\bi_k}$ and 
$\tilde\omega\mapsto T_{\bi_k}^{\tilde\omega}$ are Borel measurable. For all  
$\tilde\omega\in\widetilde\Omega$, set 
\[
Z^{\tilde\omega}(\bi)=\lim_{k\to\infty}f^{\tilde\omega}_{\bi_k}(0)\text{ and } 
A^{\tilde\omega}=\{Z^{\tilde\omega}(\bi)\mid \bi\in\Sigma^{\tilde\omega}\}.
\] 
The attractor $A^{\tilde\omega}$ is called the code tree fractal corresponding 
to $\tilde\omega$.
For $k\in\mathbb N$, $\tilde\omega\in\widetilde\Omega$ and 
$\bi\in\Sigma^{\tilde\omega}$, the 
{\it cylinder of length $k$ determined by $\bi$} is defined as
\[
[\bi_k]=\{\bj\in\Sigma^{\tilde\omega}\mid j_l=i_l\text{ for all }l=1,\dots,k\}.
\]

We proceed by recalling the definition of a neck level which is an essential 
feature of $V$-variable fractals, see for example \cite{BHS}.   
{\it A neck list} $\bN=(N_m)_{m\in\mathbb N}$ is a strictly increasing sequence of 
natural numbers. We use the notation $\Omega$ for a subset 
of $\widetilde\Omega\times\mathbb N^{\mathbb N}$ consisting of elements 
$\omega=(\tilde\omega,\bN)$ such that
\begin{itemize}
\item $\bN=(N_m)_{m\in\mathbb N}$ is a neck list and
\item if $\bi_{N_m}\bj_l,\bi'_{N_m}\in\Sigma_*^{\tilde\omega}$, then
$\bi'_{N_m}\bj_l\in\Sigma_*^{\tilde\omega}$ and
$\tilde\omega(\bi_{N_m}\bj_l)=\tilde\omega(\bi'_{N_m}\bj_l)$.
\end{itemize}
Neck levels guarantee that the attractor $A^{\tilde\omega}$ is globally nearly 
homogeneous in the sense that if $N_m\in\mathbb N$ is a neck level of 
$\tilde\omega$, then all sub code trees of $\tilde\omega$ rooted at vertices 
$\bi\in\Sigma_*^{\tilde\omega}$ with $|\bi|=N_m$ are identical.  

A function $\Xi\colon\Omega\to\Omega$ is {\it a shift }
if $\Xi(\tilde\omega,\bN)=(\hat\omega,\hat\bN)$, where
$\hat N_m=N_{m+1}-N_1$ for all $m\in\N$ and 
$\hat\omega(\bj_l)=\tilde\omega(\bi_{N_1}\bj_l)$
for all $\bj_l$ such that $\bi_{N_1}\bj_l\in\Sigma_*^{\tilde\omega}$. Note that, by 
the definition of a neck, the definition of $\hat\omega$ does not depend on 
the choice of $\bi_{N_1}$. For all $i\in\mathbb N$ and 
$\omega=(\tilde\omega,\bN)\in\Omega$, we write $N_i(\omega)=N_i$
for the projection of $\omega$ onto the $i$-th coordinate of 
$\bN$. We equip $\Omega$ with the topology
generated by {\it cylinders}
\begin{align*}
[(\tilde\omega,\bN)_m]=\{(\hat\omega,\hat\bN)\in\Omega\mid\, &
  \hat N_i=N_i\text{ for all }i\le m\text{ and }
  \hat\omega(\tau)=\tilde\omega(\tau)\\
&\text{for all }\tau\text{ with } |\tau|<N_m\}.
\end{align*}
Since $\omega\mapsto N_1(\omega)$ is continuous as a projection, the function
$\omega\mapsto N_1(\omega)$ is Borel measurable.
For any function $\phi$ of $\tilde\omega$, we use the notation $\phi(\omega)$ 
to view $\phi$ as a function of $\omega$. Finally, for 
all $n,m\in\mathbb N\cup\{0\}$ with $n<m$, define
\[
\Sigma_*^{\omega}(n,m)=\{i_{N_n(\omega)+1}\cdots i_{N_m(\omega)}\mid
  \bi_{N_n(\omega)}i_{N_n(\omega)+1}\cdots i_{N_m(\omega)}\in\Sigma_*^{\omega}\},
\]
where $N_0=0$. For all $s\ge 0$, {\it the pressure} is defined as follows
\begin{equation}\label{pressure}
p^{\omega}(s)=\lim_{k\to\infty}\frac{\log S^{\omega}(k,s)}k
\end{equation}
provided that the limit exists. Here
\[
S^{\omega}(k,s)=\sum_{\bi_k\in\Sigma_*^{\omega}}\Phi^s(T_{\bi_k}^{\omega})
\]
for all $k\in\mathbb N$. Since $T\mapsto\Phi^s(T)$ is a continuous function, 
the map $\omega\mapsto p^{\omega}(s)$ is Borel measurable.

It is well known that for affine iterated function systems
the pressure always exists. For code tree fractals this is not always the
case as explained in \cite{JJKKSS}. However, for typical random code 
tree fractals the pressure function exists and has a unique zero (see 
\cite[Theorem 4.3]{JJKKSS}).

\begin{theorem}\label{pexists}
Assume that $P$ is an ergodic $\Xi$-invariant Borel probability measure on
$\Omega$ such that $\int_{\Omega}N_1(\omega)\,dP(\omega)<\infty$. Then for 
$P$-almost all $\omega\in\Omega$, the pressure $p^{\omega}(s)$ exists for all 
$s\in[0,\infty[$. Furthermore, $p^{\omega}$ is strictly decreasing and there 
exists a unique $s_0$ such that $p^{\omega}(s_0)=0$ for $P$-almost all 
$\omega\in\Omega$.
\end{theorem}

For the purpose of identifying certain translation vectors 
(for motivation of the identification, we refer to \cite{JJKKSS,JJLS}), we 
equip the set
$\widehat\Lambda=\{(\lambda,i)\mid\lambda\in\Lambda\text{ and } 
i=1,\dots,M_\lambda\}$
with an equivalence relation $\sim$ satisfying the following assumptions:
\begin{itemize}
\item the cardinality $\mathcal A$ of the set of equivalence classes
      $\ba:=\widehat\Lambda/\sim$ is finite,
\item for every $\lambda\in\Lambda$, we have $(\lambda,i)\sim (\lambda,j)$ if
      and only if $i=j$ and
\item the equivalence classes, regarded as subsets of $\Lambda$, 
are Borel sets.
\end{itemize}
Using the relation $\sim$ for the purpose of identifying translation vectors, 
the set of equivalence classes $\ba$ may be viewed as an element
of $\mathbb R^{d\mathcal A}$. We write $A_\ba^{\omega}$ for the attractor of 
a code tree $\omega$ to emphasise the dependence on $\ba$.

For determining the almost sure value of Hausdorff dimension
for random affine code tree fractals, the affinity dimension turns out to be
useful. The affinity dimension is defined in terms of the $s$-dimensional 
net measure in the following manner. Let $s\ge0$. 
We denote by $\M^s$ the $s$-dimensional net measure defined for all Borel 
sets $E\subset\Sigma^{\omega}$ by
\[
\M^s(E)=\lim_{j\to\infty}\M_j^s(E)
\]
where 
\[
\M_j^s(E)=\inf\Big\{\sum_{\i_k\in J}\Phi^s(T^{\omega}_{\i_k})\mid 
J\subset\Sigma_*^{\omega}, E\subset\bigcup_{\i_k\in J}[\i_k] \textrm{ and }  
k\geq j\Big\}.
\]
The affinity dimension of $\Sigma^{\omega}$ is defined as
\begin{equation}\label{affinitydimdef}
\alpha^{\omega}=\inf\{s\ge0\mid\M^s(\Sigma^{\omega})=0\}
=\sup\{s\ge0\mid\M^s(\Sigma^{\omega})=\infty\}.
\end{equation}

\begin{remark}\label{affinitydim}
According to the standard definition, the affinity dimension is the unique 
zero of the pressure determined by the singular value function. Note that in 
the case when the maps are similarities, the zero of the pressure equals the 
similarity dimension of a self-similar set. In \cite{Fa}, Falconer proved that, 
for affine iterated function systems, the pressure exists and its unique zero
equals the number $\alpha^\omega$ defined in \eqref{affinitydimdef}. As pointed 
out in the introduction, for general code tree fractals the pressure need not 
exist and even if it exists, it is not necessarily equal to the number 
$\alpha^\omega$ (see \cite{JJKKSS}). Since $\alpha^\omega$ is defined for 
any code tree fractal, we prefer it as the affinity dimension. Our choice is
strongly supported by the following theorem.
\end{remark}

We denote the Hausdorff dimension by $\dimH$. According to 
\cite[Theorem 3.2]{JJKKSS}, for all $\omega\in\Omega$, Hausdorff and 
affinity dimensions of $A_\ba^{\omega}$ agree for almost all $\ba$.

\begin{theorem}\label{theorem1}
Assume that $\overline{\sigma}<1/2$. Let $\omega\in\Omega$. Then
\[
\dimH A_\ba^{\omega}=\min\{\alpha^{\omega},d\} \ \text{ for }\ 
\mathcal{L}^{d\mathcal{A}}\text{-almost all }\ba\in\mathbb R^{d\mathcal{A}}.
\]
\end{theorem}

\section{Results}\label{main}

In this section, we state and prove our main result 
(Theorem \ref{main theorem}), according to which, almost surely with respect
to any ergodic $\Xi$-invariant measure having finite expectation for the first
neck level, the Hausdorff, packing and box counting dimensions of code tree 
fractals are equal to the zero of the pressure for almost all translation 
vectors.

We proceed by verifying auxiliary results. Our first aim is to prove that the 
affinity dimension $\alpha^\omega$ is constant 
almost surely.

\begin{proposition}\label{alphaconst}
Let $P$ be an ergodic $\Xi$-invariant probability measure on $\Omega$ with 
$\int_\Omega N_1(\omega)\,dP(\omega)<\infty$. There exists $\alpha\ge 0$ 
such that $\alpha^\omega=\alpha$ for $P$-almost all $\omega\in\Omega$.
\end{proposition}

\begin{proof}
Since $\int_{\Omega}N_1(\omega)\,dP(\omega)<\infty$, we have 
$N_1(\omega)<\infty$ for $P$-almost all $\omega\in\Omega$. For all such 
$\omega\in\Omega$, the definition of $\M_n^s$ implies that 
\[
(\underline\sigma^s)^{N_1(\omega)}\M_n^s(\Sigma^{\Xi(\omega)})\le
 \M_{N_1(\omega)+n}^s(\Sigma^\omega)
 \le (M\overline\sigma^s)^{N_1(\omega)}\M^s_n(\Sigma^{\Xi(\omega)})
\]
for all $s\ge0$ and $n\in\mathbb N$. Letting $n\to\infty$, gives
\begin{equation}\label{equation section2 1}
(\underline\sigma^s)^{N_1(\omega)}\M^s(\Sigma^{\Xi(\omega)})\le\M^s(\Sigma^\omega)
\le (M\overline\sigma^s)^{N_1(\omega)}\M^s(\Sigma^{\Xi(\omega)}).
\end{equation}
From \eqref{equation section2 1}, we deduce that the set 
\[
E(s):=\{\omega\in\Omega\mid\M^s(\Sigma^\omega)>0\}
\]
is $\Xi$-invariant, that is, $\Xi^{-1}((E(s))=E(s)$. Since $\Sigma^\omega$ is
compact and cylinder sets are open, one may use finite coverings when 
calculating $\M_j^s(\Sigma^\omega)$. Thus, the Borel measurability of 
$\omega\mapsto\Phi^s(T_{\bi_k}^\omega)$ implies that 
$\omega\mapsto\M^s(\Sigma^\omega)$ is a Borel map and, therefore, $E(s)$ is 
a Borel set. Since $P$ is ergodic, for all $s\ge 0$, 
$P(E(s))$ is either 0 or 1. It
follows that $\alpha^\omega$ is a constant for $P$-almost all $\omega\in\Omega$, 
since otherwise there exists $s\geq 0$ such that $0<P(E(s))<1$.
\end{proof}

Now we introduce another net measure $\widetilde{\M}^s$ on $\Sigma^\omega$ which
is similar to the natural net measure $\M^s$ but takes into account the 
neck structure. For all $n\in\mathbb N$,
let 
\[
I_n=\{\i_{N_k(\omega)}\in\Sigma^\omega_*\mid k\ge n\}.
\]
For $s\ge 0$, {\it the $s$-dimensional neck net measure} $\widetilde{\M}^s$ is 
defined for all Borel subsets $E$ of $\Sigma^\omega$ by
\[
\widetilde{\M}^s(E)=\lim_{j\to\infty}\widetilde{\M}_j^s(E)
\]
where 
\[
\widetilde{\M}_j^s(E)=\inf\Bigl\{\sum_{\i\in I}\Phi^s(T^\omega_{\i})
\mid E\subset\bigcup_{\i\in I}[\i]\text{ and }\ I\subset I_j \Bigr\}.
\]
We define {\it the neck affinity dimension}  of $\Sigma^\omega$ as
\begin{equation}\label{neckaffdim}
\tilde{\alpha}^\omega =\inf\{s\ge0\mid\widetilde{\M}^s(\Sigma^\omega)=0\}
=\sup\{s\ge 0\mid\widetilde{\M}^s(\Sigma^\omega)=\infty\}.
\end{equation}
Since in the definition of $\widetilde{\M}^s(E)$ there are more restrictions 
on possible coverings of $E$ than in the case of $\M^s(E)$, we have
$\M^s(E)\le\widetilde{\M}^s(E)$. Hence,
\begin{equation}\label{equation section2 2}
\alpha^\omega\leq \tilde{\alpha}^\omega\text{ for all }\omega\in\Omega.
\end{equation}

It turns out that $\alpha^\omega$ and $\tilde{\alpha}^\omega$ are equal 
almost surely.

\begin{proposition}\label{proposition equality of alpha and alpha tilde}
Letting $P$ be an ergodic $\Xi$-invariant probability measure on $\Omega$ with 
$\int_\Omega N_1(\omega)\,dP(\omega)<\infty$, we have that
$\alpha^\omega=\tilde{\alpha}^\omega$ for $P$-almost all $\omega\in\Omega$.
\end{proposition}

\begin{proof}
Observe that there exists a Borel set $F\subset \Omega$ with $P(F)=1$ such 
that for all $\omega\in F$ and $\delta>0$, we have 
\begin{equation}\label{equation section2 2.1}
N_{n+1}(\omega)-N_{n}(\omega)\leq \delta N_{n}(\omega)
\end{equation}
for sufficiently large $n\in\N$. Indeed, since 
\begin{equation}\label{N1orbit}
N_{n}(\omega)=\sum_{k=0}^{n-1}\bigl(N_{k+1}(\omega)-N_{k}(\omega)\bigr)
=\sum_{k=0}^{n-1}N_1(\Xi^k(\omega)),
\end{equation}
Birkhoff's ergodic theorem gives
\[
\lim_{n\to\infty}\frac{N_n(\omega)}{n}
=\lim_{n\to\infty}\frac{1}{n}\sum_{k=0}^{n-1}N_1(\Xi^k(\omega))
=\int_{\Omega}N_1(\omega)\,dP(\omega)
\]
for $P$-almost all $\omega\in\Omega$. This leads to
\[
\lim_{n\to\infty}\frac{N_{n+1}(\omega)-N_n(\omega)}{n}=0
=\lim_{n\to\infty}\frac{N_{n+1}(\omega)-N_n(\omega)}{N_n(\omega)}
\]
for $P$-almost all $\omega\in\Omega$,
completing the proof of \eqref{equation section2 2.1}.

Now, by \eqref{equation section2 2}, it suffices to verity that 
\begin{equation}\label{aim}
\alpha^\omega\ge\tilde{\alpha}^\omega\text{ for }P\text{-almost all }\omega\in F. 
\end{equation}
For this purpose, consider $\omega\in F$ and $\varepsilon>0$. We will show 
that there exist a finite Borel measure $\mu^\omega$ on $\Sigma^\omega$ and 
a constant $c(\omega)$ such that 
\begin{equation}\label{equation section2 3}
\mu^\omega[\i_k]\leq c(\omega)\Phi^{ \tilde{\alpha}^\omega-\vep}(T^\omega_{\i_k})
\end{equation}
for all cylinders $[\i_k]$ with $k\geq1$. It is shown in the proof of
\cite[Theorem 3.2]{JJKKSS} that this results in
\[
\dimH A^\omega_{\bf a}\geq \tilde{\alpha}^\omega-\varepsilon\text{ for } 
\mathcal{L}^{d\mathcal{A}}\text{-almost all } {\bf a}\in\mathbb R^{d\mathcal{A}}.
\]
From Theorem \ref{theorem1}, we conclude that 
$\alpha^\omega\geq \tilde{\alpha}^\omega-\varepsilon$ for $P$-almost all 
$\omega\in F$. Taking a sequence $\varepsilon_i$ tending to zero, completes 
the proof of \eqref{aim}. 

It remains to prove \eqref{equation section2 3}.
Since $\widetilde{\M}^{ \tilde{\alpha}^\omega-\varepsilon/2}(\Sigma^\omega)=\infty$, 
we derive, similarly as in \cite[Proposition 2.8]{RV} (the proof of 
\cite[Proposition 2.8]{RV} is written for homogeneous code trees but it works 
for inhomogeneous code trees as well), that there exist a finite Borel measure 
$\mu^\omega$ on $\Sigma^\omega$ and a constant $c'(\omega)$ such that
\begin{equation}\label{equation section2 4}
\mu^\omega[\i_{N_k(\omega)}]
\leq c'(\omega)\Phi^{ \tilde{\alpha}^\omega-\varepsilon/2}(T^\omega_{\i_{N_k(\omega)}})
\end{equation}
for all cylinders $[\i_{N_k(\omega)}]$ with $k\geq1$.
We aim to verify that $\mu^\omega$ satisfies \eqref{equation section2 3}. 
For all $n\in\N$ , there  is a unique $k\in\mathbb N$ such that 
$N_k(\omega)\leq n<N_{k+1}(\omega)$. By \eqref{equation section2 4} and the 
submultiplicativity of $\Phi^s$, we get
\begin{align}
\mu^\omega[\i_n]&=\sum_{i_{n+1},\dots,i_{N_{k+1}(\omega)}}
 \mu^\omega[\i_n i_{n+1}\cdots i_{N_{k+1}(\omega)}]\label{equation section2 5}\\
&\leq c'(\omega)\Phi^{ \tilde{\alpha}^\omega-\varepsilon/2}(T^\omega_{\i_{n}})
 \sum_{i_{n+1},\dots,i_{N_{k+1}(\omega)}}\Phi^{ \tilde{\alpha}^\omega-\varepsilon/2}
 (T^\omega_{i_{n+1}}\cdots T_{i_{N_{k+1}(\omega)}}^\omega)\nonumber\\
&=c'(\omega)\Phi^{ \tilde{\alpha}^\omega-\varepsilon}
 (T^\omega_{\i_{n}})\frac{\Phi^{ \tilde{\alpha}^\omega-\varepsilon/2}(T^\omega_{\i_{n}})}
 {\Phi^{ \tilde{\alpha}^\omega-\varepsilon}(T^\omega_{\i_{n}})}
 \sum_{i_{n+1},\dots,i_{N_{k+1}(\omega)}}\Phi^{\tilde{\alpha}^\omega-\varepsilon/2}
 (T^\omega_{i_{n+1}}\cdots T_{i_{N_{k+1}(\omega)}}^\omega).\nonumber
\end{align}
The assumptions \eqref{equation introduction 1} and 
\eqref{equation introduction 2} imply that
\[
a:=\sum_{i_{n+1},\dots,i_{N_{k+1}(\omega)}}\Phi^{ \tilde{\alpha}^\omega-\varepsilon/2}
  (T^\omega_{i_{n+1}}\cdots T_{i_{N_{k+1}(\omega)}}^\omega)\leq M^{N_{k+1}(\omega)-n}
  \leq M^{N_{k+1}(\omega)-N_k(\omega)}
\]
and 
\[
b:=\frac{\Phi^{ \tilde{\alpha}^\omega-\varepsilon/2}(T^\omega_{\i_{n}})}
{\Phi^{ \tilde{\alpha}^\omega-\varepsilon}(T^\omega_{\i_{n}})}
\leq\overline{\sigma}^{\varepsilon n/2}\leq\overline{\sigma}^{\varepsilon N_k(\omega)/2}.
\]
As $\omega \in F$, for every $\delta>0$, we have 
$N_{k+1}(\omega)-N_k(\omega)\le\delta N_k(\omega)$ 
when $k$ is large enough.  Since $\overline{\sigma}<1$ and $M<\infty$, 
taking $\delta$ sufficiently small, gives $ab\leq 1$ for all sufficiently 
large $k\in\N$. Substituting the product of $a$ and $b$ with 1 in 
\eqref{equation section2 5}, we obtain that
\[
\mu^\omega[\i_n]\leq c'(\omega)\Phi^{ \tilde{\alpha}^\omega-\vep}(T^\omega_{\i_n})
\]
for large enough $n\in\N$. Thus we can find a constant $c(\omega)$ such that 
\[
\mu^\omega[\i_n]\leq c(\omega)\Phi^{ \tilde{\alpha}^\omega-\vep}(T^\omega_{\i_n})
\]
for all cylinders $[\i_n]$ with $n\geq 1$.
\end{proof}

When calculating $\widetilde{\M}^s(\Sigma^\omega)$ for $s>\tilde\alpha^\omega$,
by compactness of $\Sigma^\omega$, one may find a finite covering 
$\{[\bi_{N_k}]\}_{\bi_{N_k}\in J}$ of $\Sigma^\omega$ such that 
$\sum_{\bi_{N_k}\in J}\Phi^s(T_{\bi_{N_k}}^\omega)<1$. In particular, the lengths of
$\bi_{N_k}$ are bounded but the bound may depend on $\omega\in\Omega$.  
Next we prove a lemma which states that, apart from a small exceptional set,
one can find a uniform bound. In what follows we need the following notation.
For all $n\in\N$ and $\omega\in\Omega$, set
\[
J_n(\omega)=\bigcup_{k=1}^n\{\i_{N_k(\omega)}\mid\i_{N_k(\omega)}\in\Sigma_*^\omega\}
\text{ and } J(\omega)=\bigcup_{k=1}^\infty J_k(\omega).
\]

\begin{lemma}\label{Bepsilon}
Let $P$ be an ergodic $\Xi$-invariant probability measure on $\Omega$ 
satisfying $\int_\Omega N_1(\omega)\,dP(\omega)<\infty$, and let $s>\alpha$, 
where $\alpha$ is as in Proposition~\ref{alphaconst}. For every 
$\varepsilon>0$, there exist a Borel set $B(\varepsilon)\subset\Omega$ with 
$P(B(\varepsilon))\geq 1-\varepsilon$ and $R\in\N$ such that for every 
$\omega\in B(\varepsilon)$, there is $C_R(\omega)\subset J_R(\omega)$ with 
\[
\Sigma^\omega\subset\bigcup_{\i_{N_k(\omega)}\in C_R(\omega)}[\i_{N_k(\omega)}]\text{ and }
\sum_{\i_{N_k(\omega)}\in C_R(\omega)}\Phi^s(T^\omega_{\i_{N_k(\omega)}})<1.
\]
\end{lemma}

\begin{proof}
We introduce a family of functions $\{f_n\}_{n\in\mathbb N}$ defined on $\Omega$ 
by 
\[
f_n(\omega)=\min\Bigl\{\sum_{\i_{N_k(\omega)}\in I}\Phi^s(T^\omega_{\i_{N_k(\omega)}})
\mid\Sigma^\omega\subset\bigcup_{\i_{N_k(\omega)}\in I}[\i_{N_k(\omega)}]\text{ and }
I\subset J_n(\omega)\Bigr\}.
\]
For all $n\in\N$, the function $f_n(\omega)$ is Borel measurable. Indeed, this
follows from the fact that 
$\omega\mapsto\sum_{\i_{N_k(\omega)}\in I}\Phi^s(T^\omega_{\i_{N_k(\omega)}})$ is a Borel
function as a finite sum of Borel functions.
 
Recalling that $\tilde{\alpha}^\omega=\alpha$ for $P$-almost all 
$\omega\in\Omega$ (see Propositions~\ref{alphaconst} and 
\ref{proposition equality of alpha and alpha tilde}), gives
$\widetilde{\M}^s(\Sigma^\omega)=0$ for $P$-almost all $\omega\in\Omega$. 
In particular, for $P$-almost all $\omega\in\Omega$ and for every $\delta>0$, 
there exists a finite set $I\subset J(\omega)$ such that
\[
\sum_{\i_{N_k(\omega)}\in I}\Phi^s(T^\omega_{\i_{N_k(\omega)}})<\delta
\text{ and }\Sigma^\omega\subset\bigcup_{\i_{N_k(\omega)}\in I}[\i_{N_k(\omega)}]. 
\]
Recall that we may use finite coverings in the definition 
of $\widetilde{\M}^s(\Sigma^\omega)$ because $\Sigma^\omega$ is compact and 
every cylinder in $\Sigma^\omega_*$ is an open set.
This implies that for $P$-almost all $\omega\in\Omega$, we have 
\[
\lim_{n\to\infty}f_n(\omega)=0.
\]
By Egorov's theorem, for every $\varepsilon >0$, there exists a Borel set
$B(\varepsilon)\subset\Omega$ with $P(B(\varepsilon))\geq 1-\varepsilon$  such 
that $f_n(\omega)$ converges uniformly to 0 on $B(\varepsilon)$. In particular,
there exists $R\in\N$ such that $f_R(\omega)<1$ for all 
$\omega\in B(\varepsilon)$.
\end{proof}

Our plan is to show that for any $s>\alpha$, we have $s\ge s_0$, that 
is, $p^\omega(s)\le 0$. To achieve this goal, we will employ 
Lemma~\ref{Bepsilon} and, for all large enough $L\in\mathbb N$, we
decompose the finite code tree
\[
\Sigma^\omega(L)=\{\i_{N_L(\omega)}\mid\i_{N_L(\omega)}\in\Sigma^\omega_*\}
\]
into subtrees $C_R(\omega')$ given by Lemma~\ref{Bepsilon} for some suitable
$\omega'\in\Omega$. Since the complement of the set $B(\vep)$ in 
Lemma~\ref{Bepsilon} may have positive measure, the above 
decomposition cannot cover the whole tree $\Sigma^\omega(L)$. However, the
ergodicity of $P$ guarantees that the contribution of the remaining part is
not too large. 

Now we formalise the above idea. Consider $s>\alpha$ and $\vep>0$. Let 
$B(\vep)\subset\Omega$ and $R\in\N$ be as in Lemma~\ref{Bepsilon}. Set
\[
C(\varepsilon)=\{C_R(\omega)\mid\omega\in B(\varepsilon)\},
\]
and for all $L\in\N$ with $L\ge R$, define
\begin{align*}
H(\varepsilon,L)=&\Bigl\{\{i_{N_k(\omega)}\cdots i_{N_{k'}(\omega)}|i_{N_k(\omega)}
  \cdots i_{N_{k'(\omega)}}\in\Sigma_*^{\Xi^k(\omega)}\} \mid k<L,\Xi^k(\omega)
  \notin B(\varepsilon)\text{ and }\\
 &\phantom{\Bigl\{}k'=\min\bigl\{L,\min\{j>k\mid\Xi^j(\omega)\in 
  B(\varepsilon)\}\bigr\}
  \Bigr\}.
\end{align*} 
Note that each element in $C(\varepsilon)$ is a collection of words with
possibly varying length at most $R$ while every element of $H(\vep,L)$ is a 
collection of words with same length, which is at most $N_L(\omega)$. Each 
element in $C(\varepsilon)$ and $H(\varepsilon,L)$ is called a 
subtree. For every $\omega\in\Omega$ and $L\ge R$, we decompose 
$\Sigma^\omega(L)$ into subtrees belonging to $C(\varepsilon)$ and 
$H(\varepsilon,L)$. The decomposition will be done inductively as follows.

{\it Decomposition step 1.} If $\omega\in B(\varepsilon)$, then the first 
generation decomposition subtree $D_1(\omega)$ is $C_R(\omega)$ given by 
Lemma~\ref{Bepsilon}. Otherwise, let 
\[
m(\omega)=\min\bigl\{L,\min\{k\ge 1\mid\Xi^k(\omega)\in B(\varepsilon)\}\bigr\}
\]
and define $D_1(\omega)$ as
$\{\i_{N_{m(\omega)}(\omega)}
\mid\i_{N_{m(\omega)}(\omega)}\in\Sigma^\omega_*\}\in H(\vep,L)$.
In both cases, the words $\i_{N_k(\omega)}\in D_1(\omega)$ satisfy $k\le L$ and,
moreover, the 
cylinders $[\i_{N_k(\omega)}]$ are disjoint and cover the whole space 
$\Sigma^\omega$. 

{\it Decomposition step 2.} We apply the decomposition step 1 to the 
descendants of each $\i_{N_k(\omega)}\in D_1(\omega)$ with $k\le L-R$.
More precisely, 
for every $\i_{N_k(\omega)}\in D_1(\omega)$ with $k\le L-R$, we set 
$\omega'=\Xi^k(\omega)$ and apply decomposition step 1 with $\omega$ replaced
by $\omega'$ and with the 
modification that in the definition of $m(\omega')$ $L$ is replaced by $L-k$.
In this way, every $\i_{N_k(\omega)}\in D_1(\omega)$ with $k\le L-R$ defines a 
second generation decomposition subtree $D_2(\omega,\i_{N_k(\omega)})$.
We continue this decomposition process inductively until all the nodes in
$\Sigma^\omega(L)$ which are not covered lie between levels $N_{L-R}(\omega)$ and
$N_L(\omega)$.
In this manner, we obtain a tree, denoted by $A(L)$, consisting of subtrees 
belonging to $C(\varepsilon)$ or $H(\varepsilon,L)$.

Now we are ready to prove the main technical lemma of this paper. 
We denote the characteristic function of a set $A$ by ${\bf 1}_A$ and the 
complement of $A$ by $A^c$.

\begin{lemma}\label{gooddecay}
Consider $s>\alpha$, $\vep>0$ and $\omega\in\Omega$. 
Let $B(\vep)$ and $R\in\N$ be 
as in Lemma~\ref{Bepsilon}. For every $L\ge R$, we have
\begin{equation}\label{star}
\sum_{\i_{N_L(\omega)}\in\Sigma^\omega_*}\Phi^s(T^\omega_{\i_{N_L(\omega)}})
  \le M^{Q_L(\varepsilon)+N_L(\omega)-N_{L-R}(\omega)}
\end{equation} 
where $Q_L(\varepsilon)=\sum_{k=0}^{L-1}(N_{k+1}(\omega)-N_k(\omega))
{\bf 1}_{B(\varepsilon)^c}(\Xi^k(\omega))$. 
\end{lemma}

\begin{proof} 
Recalling that in the tree $A(L)$ every word has length between 
$N_{L-R}(\omega)$  and $N_L(\omega)$ and utilising the submultiplicativity of 
$\Phi^s$ and the inequality $\overline\sigma<1$,
it follows easily that 
\[
\sum_{\i_{N_L(\omega)}\in\Sigma^\omega_*}\Phi^s(T^\omega_{\i_{N_L(\omega)}})
\le\Bigl(\sum_{\i\in A(L)}\Phi^s(T^\omega_{\i})\Bigr)M^{N_L(\omega)-N_{L-R}(\omega)}.
\] 
Hence, we only need to show that 
\begin{equation}\label{equation section2 9}
\sum_{\i\in A(L)}\Phi^s(T^\omega_{\i})\leq M^{Q_L(\varepsilon)}.
\end{equation}

Let $\Gamma$ be any subtree of $A(L)$ consisting of the subtrees obtained in
the decomposition process. The first and the last neck 
levels of $\Gamma$ are denoted by $F(\Gamma)$ and $L(\Gamma)$, respectively. 
See Figure \ref{figure 1} for an illustration of $F(\Gamma)$ and $L(\Gamma)$ 
of a subtree $\Gamma$. 
\begin{figure}[htbp]
\centering
\begin{tikzpicture}[edge from parent, sibling distance=15mm, 
level distance=15mm,
every node/.style={fill=black!30,rounded corners},
edge from parent/.style={-,thick,draw}]%
\node[fill=white] at (5,0) {level $N_k$: $k=F(\Gamma)$};
\draw[dotted] (.2,0) -- (3.1,0);
\node[fill=white] at (5,-.5) {(The first neck level of $\Gamma$.)};
\node[fill=white] at (5,-6) {level $N_{k^{\prime}}$: $k^{\prime}=L(\Gamma)$};
\draw[dotted] (1.6,-6) -- (3.1,-6);
\node[fill=white] at (5,-6.5) {(The last neck level of $\Gamma$.)};
\node[fill=white] at (-.75, -1.5) {. . .};
\node[fill=white] at (.75, -1.5) {. . .};
\node[fill=white] at (-1.5, -3) {. . .};
\node[fill=white] at (1.5, -3) {. . .};
\node[fill=white] at (-1.5, -4.5) {. . .};
\node[fill=white] at (1.5, -4.5) {. . .};
\draw[thick,|<->|] (-4.5,0) -- (-4.5,-6);
\node[fill=white] at (-4.5,-3) {\rotatebox{90}{The subtree $\Gamma$}};
\node at (0,0) {}
child {node {}}
child {node {}}
child {node {}};
\draw[loosely dashed] (-1.7,-1.7) -- (-2.8,-2.8);
\draw[loosely dashed] (0,-1.7) -- (0,-3);
\draw[loosely dashed] (1.7,-1.7) -- (2.8,-2.8);
\node at (-3,-3) {}
child {node {}}
child {node {}};
\node at (0,-3) {}
child {node {}}
child {node {}
child {node {}}
child {node {}}
};
\node at (3,-3) {}
child {node {}}
child {node {}};
\end{tikzpicture}
\caption{A subtree $\Gamma$ with its first and last neck levels.}
\label{figure 1}
\end{figure}
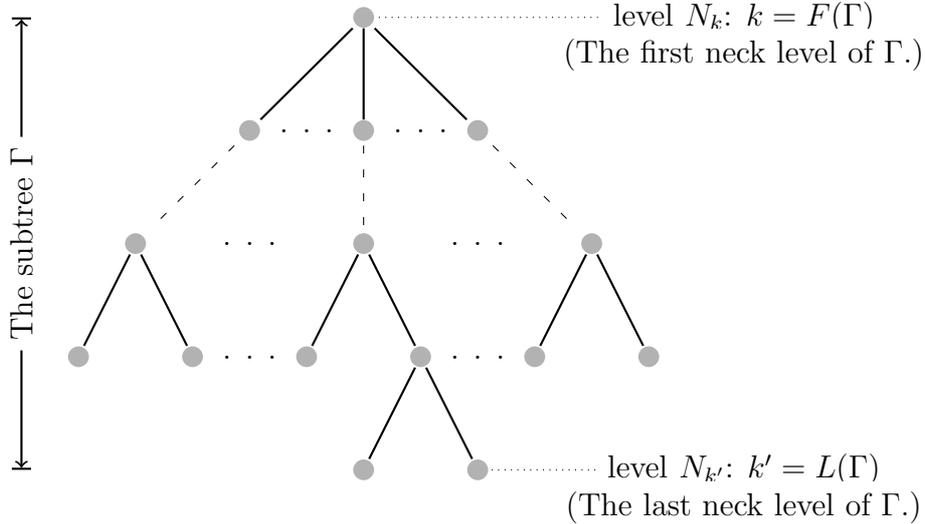
We are going to prove that
\begin{equation}\label{equation section2 10}
\sum_{\i\in \Gamma}\Phi^s(T^\omega_{\i})\leq 
 M^{\sum_{k=F(\Gamma)}^{L(\Gamma)-1}(N_{k+1}(\omega)-N_k(\omega)){\bf 1}_{B(\varepsilon)^c}(\Xi^k(\omega))}.
\end{equation}
Applying (\ref{equation section2 10}) to the subtree 
$\Gamma=A(L)$, gives (\ref{equation section2 9}). 
We prove  (\ref{equation section2 10}) by induction on the number of the 
subtrees in $\Gamma$ given by the decomposition process. 

{\it Initial step:} Assume that the subtree $\Gamma$ consists of one single 
subtree belonging to $C(\varepsilon)$ or $H(\varepsilon,L)$.
If $\Gamma\in C(\varepsilon)$, the definition of $C(\varepsilon)$ results in
\[
\sum_{\i\in\Gamma}\Phi^s(T^\omega_\i)\leq 1
\leq M^{\sum_{k=F(\Gamma)}^{L(\Gamma)-1}(N_{k+1}(\omega)-N_k(\omega)){\bf 1}_{B(\vep)^c}(\Xi^k(\omega))}.
\]
On the other hand, if $\Gamma\in H(\varepsilon,L)$, then 
$\Gamma=\{i_{N_k(\omega)}\cdots i_{N_{k'}(\omega)}\mid i_{N_k(\omega)}\cdots 
  i_{N_{k'}(\omega)}\in\Sigma^{\Xi^k(\omega)}_*\}$ 
with $\Xi^j(\omega)\notin B(\varepsilon)$ for $j=k,\dots,k'-1$, and 
$\Xi^{k'}(\omega)\in B(\varepsilon)$ or $k'=L$. Hence,
\[
\sum_{\i\in\Gamma}\Phi^s(T^\omega_\i)\le\sum_{\i\in\Gamma}1
  \le M^{N_{L(\Gamma)}(\omega)-N_{F(\Gamma)}(\omega)}
=M^{\sum_{k=F(\Gamma)}^{L(\Gamma)-1}(N_{k+1}(\omega)-N_k(\omega)){\bf 1}_{B(\vep)^c}(\Xi^k(\omega))}.
\]
We deduce that \eqref{equation section2 10} holds when $\Gamma$ contains only
one subtree.

{\it Inductive step $n$:} Suppose that \eqref{equation section2 10} is true 
for every subtree containing at most $n-1$ decomposition subtrees. Letting 
$\Gamma$ be a subtree containing $n$ decomposition subtrees, we show that 
\eqref{equation section2 10} holds for $\Gamma$. For this purpose, 
we denote by $\Gamma(1)$ the subtree of the smallest generation in $\Gamma$. 
Recall that $\Gamma(1)$ belongs to $C(\varepsilon)$ or $H(\varepsilon,L)$. 
Suppose that $\Gamma(1)=\{\i_1,\i_2,\dots,\i_\gamma\}$. For 
$\ell=1,\dots,\gamma$, let $\Gamma_\ell$ be the subtree of $\Gamma$ rooted at 
$\i_\ell$. See Figure~\ref{figure 2} for an illustration of $\Gamma(1)$ and 
$\Gamma_\ell$.
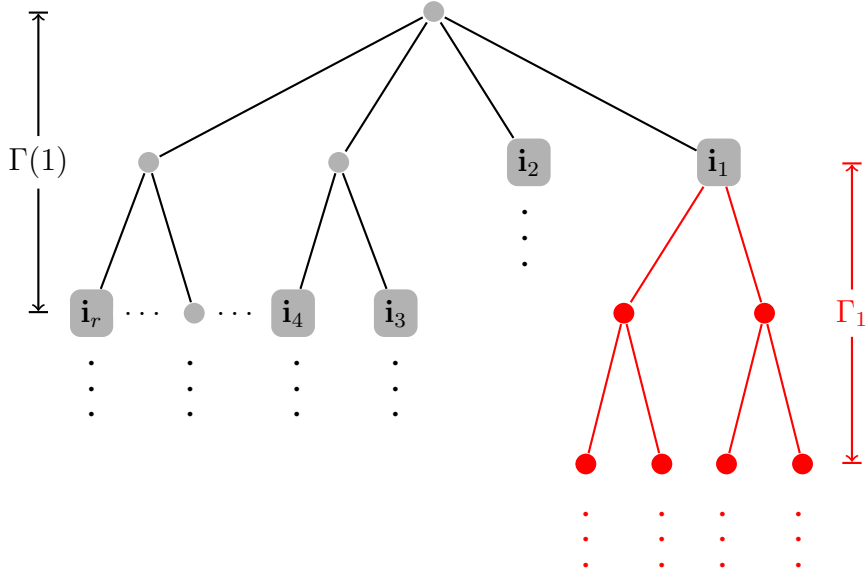
\begin{figure}[htbp]
\centering
\begin{tikzpicture}[level distance=20mm,%
level 1/.style={sibling distance=25mm}, %
level 2/.style={sibling distance=12mm}, %
level 3/.style={sibling distance=10mm}, %
every node/.style={fill=black!30,rounded corners},%
edge from parent/.style={-,thick,draw}]%
\node at (0,0) {}
child {node {}
	child[sibling distance=15mm] {node {$\mathbf{i}_r$}}
	child {node {}}}
child {node {}
	child {node {$\mathbf{i}_4$}}
	child[sibling distance=15mm] {node {$\mathbf{i}_3$}}}
child {node {$\mathbf{i}_2$}}
child {node {$\mathbf{i}_1$}
	child[red, sibling distance=25mm] {node[red] {}
		child{node[red] {}}
		child{node[red] {}}}
	child[red] {node[red] {}
		child{node[red] {}}
		child{node[red] {}}}};
\draw[thick,|<->|] (-5.2,0) -- (-5.2,-4);
\node[fill=white] at (-5.2,-2) {$\Gamma(1)$};		
\draw[red,thick,|<->|] (5.5,-2) -- (5.5,-6);
\node[red, fill=white] at (5.5,-4) {$\Gamma_1$};
\node[fill=white] at (1.2,-3) {\rotatebox{90}{\bf{. . .}}};
\node[fill=white] at (-2.58,-4) {$\dots$};
\node[fill=white] at (-3.8,-4) {$\dots$};
\foreach \x in {-0.5,-1.8, -3.2, -4.5}
	{\node[fill=white] at (\x,-5) {\rotatebox{90}{\bf{. . .}}};}
\foreach \x in {2,3, 3.8, 4.8}
	{\node[red,fill=white] at (\x,-7) {\rotatebox{90}{\bf{. . .}}};}
		
\end{tikzpicture}
\caption{The first generation $\Gamma(1)$ of $\Gamma$ (in black) and 
the subtree $\Gamma_1$ rooted at $\i_1$ (in red).}
\label{figure 2}
\end{figure}

Let 
\[
F(\Gamma_{\ell_0})=\min_{1\leq \ell\leq \gamma}\{F(\Gamma_\ell)\}
\text{ and }L(\Gamma_{\ell_1})=\max_{1\leq \ell\leq \gamma}\{L(\Gamma_\ell)\}.
\]
Using the fact that, for all $\ell=1,\dots,\gamma$, the subtree $\Gamma_\ell$ 
contains less than $n$ decomposition subtrees, and applying the induction 
hypothesis for $\Gamma_\ell$ in the first inequality below, gives 
\begin{equation}\label{equation section2 12}
\begin{split}
\sum_{\i\in\Gamma_\ell}\Phi^s(T^\omega_\i)
&\le M^{\sum_{k=F(\Gamma_\ell)}^{L(\Gamma_\ell)-1}(N_{k+1}(\omega)-N_k(\omega))
    {\bf 1}_{B(\vep)^c}(\Xi^k(\omega))}\\
&\le M^{\sum_{k=F(\Gamma_{\ell_0})}^{L(\Gamma_{\ell_1})-1}(N_{k+1}(\omega)-N_k(\omega))
    {\bf 1}_{B(\vep)^c}(\Xi^k(\omega))}.
\end{split}
\end{equation}

Now there are two cases to consider: $\Gamma(1)\in C(\varepsilon)$ or 
$\Gamma(1)\in H(\varepsilon,L)$.
To begin with, suppose that $\Gamma(1)\in C(\varepsilon)$. Then
$\sum_{\i\in \Gamma(1)}\Phi^s(T^\omega_\i)\leq1$ and, therefore, by the
submultiplicativity of $\Phi^s$ and by \eqref{equation section2 12}, we obtain
\begin{align*}
\sum_{\i\in\Gamma}&\Phi^s(T^\omega_\i)
  \leq\sum_{\ell=1}^\gamma\Big(\Phi^s(T^\omega_{\i_\ell})
  \sum_{\i\in\Gamma_\ell}\Phi^s(T^\omega_\i)\Big)
  \le \max_{1\leq \ell\leq \gamma}\sum_{\i\in\Gamma_\ell}\Phi^s(T^\omega_\i)\\ 
&\le M^{\sum_{k=F(\Gamma_{\ell_0})}^{L(\Gamma_{\ell_1})-1}(N_{k+1}(\omega)-N_k(\omega))
  {\bf 1}_{B(\varepsilon)^c}(\Xi^k(\omega))}
  \le M^{\sum_{k=F(\Gamma)}^{L(\Gamma)-1}(N_{k+1}(\omega)-N_k(\omega)){\bf 1}_{B(\vep)^c}(\Xi^k(\omega))}.
\end{align*}
Hence, \eqref{equation section2 10} holds for $\Gamma$.

Finally, assume that $\Gamma(1)\in H(\varepsilon,L)$. Again, 
by the submultiplicativity of $\Phi^s$ and \eqref{equation section2 12}, we have
\begin{equation}\label{equation section2 14}
\begin{split}
\sum_{\i\in\Gamma}\Phi^s(T^\omega_\i)&
 \le\sum_{\ell=1}^\gamma\Big(\Phi^s(T^\omega_{\i_\ell})
  \sum_{\i\in\Gamma_\ell}\Phi^s(T^\omega_\i)\Big)\\
 &\le\Big(\sum_{\ell=1}^\gamma\Phi^s(T^\omega_{\i_\ell})\Big)
  M^{\sum_{k=F(\Gamma_{\ell_0})}^{L(\Gamma_{\ell_1})-1}(N_{k+1}(\omega)-N_k(\omega))
  {\bf 1}_{B(\vep)^c}(\Xi^k(\omega))}.
\end{split}
\end{equation}
Since $\Gamma(1)\in H(\varepsilon,L)$, all the words in $\Gamma(1)$ have the 
same length and, therefore, $F(\Gamma_{\ell_0})=L(\Gamma(1))$. Furthermore, 
$\Xi^j(\omega)\notin B(\omega)$ for $j=0,\dots,L(\Gamma(1))-1$. Hence,  
\begin{equation}\label{equation section2 15}
\begin{split}
\sum_{\ell=1}^\gamma\Phi^s(T^\omega_{\i_\ell})
&\le M^{\sum_{k=F(\Gamma(1))}^{L(\Gamma(1))-1}(N_{k+1}(\omega)-N_k(\omega))
   {\bf 1}_{B(\vep)^c}(\Xi^k(\omega))}\\
&=M^{\sum_{k=F(\Gamma)}^{F(\Gamma_{\ell_0})-1}(N_{k+1}(\omega)-N_k(\omega))
   {\bf 1}_{B(\vep)^c}(\Xi^k(\omega))}.
\end{split}
\end{equation}
Employing \eqref{equation section2 15}  in  \eqref{equation section2 14} 
and observing that $L(\Gamma_{\ell_1})\leq L(\Gamma)$, yields
\[
\sum_{\i\in\Gamma}\Phi^s(T^\omega_\i)
\le M^{\sum_{k=F(\Gamma)}^{L(\Gamma)-1}(N_{k+1}(\omega)-N_k(\omega)){\bf 1}_{B(\vep)^c}(\Xi^k(\omega))}.
\]
We conclude that \eqref{equation section2 10} is true for $\Gamma$.
\end{proof}

Now we are ready to state and prove our main theorem. We denote the packing
and box counting dimensions by $\dimp$ and $\dimb$, respectively.

\begin{theorem}\label{main theorem} Assume that $\overline{\sigma}<1/2$ and
$P$ is an ergodic $\Xi$-invariant probability measure on $\Omega$ with 
$\int_\Omega N_1(\omega)\,dP(\omega)<\infty$. Then for $P$-almost all 
$\omega\in\Omega$,
\[
\dimH A_\ba^\omega=\dimp A_\ba^\omega=\dimb A_\ba^\omega=\min\{s_0,d\}\text{ for }
 \mathcal{L}^{d\mathcal{A}}\text{-almost all }\ba\in\mathbb R^{d\mathcal{A}}
\] 
where $s_0$ is the unique zero of the pressure (see Theorem~\ref{pexists}).
\end{theorem} 

\begin{proof}
Noting that the proof of the upper bound $\dimb A_\ba^\omega\le\min\{s_0,d\}$ 
given in \cite[Theorem 5.1]{JJKKSS} is valid under the assumptions of 
Theorem~\ref{main theorem}, we may restrict our consideration to the 
Hausdorff dimension. By Theorem~\ref{theorem1} and 
Proposition~\ref{alphaconst}, it is sufficient to verify that $s_0=\alpha$. 

Let $s>s_0$. By Theorem~\ref{pexists}, the pressure is strictly decreasing and,
therefore, $p^\omega(s)<0$ for $P$-almost all $\omega\in\Omega$. Considering such
$\omega\in\Omega$ and using the definition of the pressure \eqref{pressure}, 
we conclude that 
$\sum_{\i_{N_k(\omega)}\in\Sigma_*^\omega}\Phi^s(T_{\i_{N_k(\omega)}}^\omega)\le 1$ 
for all sufficiently large $k\in\N$. Recalling 
Proposition~\ref{proposition equality of alpha and alpha tilde}, the 
definition of the neck affinity dimension \eqref{neckaffdim} implies that 
$\alpha=\tilde\alpha^\omega\le s$ for $P$-almost all $\omega\in\Omega$. This 
completes the proof of the fact that $\alpha\le s_0$.   

For the purpose of proving the opposite inequality, consider $s>\alpha$ and 
$\varepsilon>0$. Let $B(\varepsilon)\subset\Omega$ be as in 
Lemma~\ref{Bepsilon}. Applying the Birkhoff's ergodic theorem and recalling
\eqref{N1orbit}, implies that for $P$-almost all $\omega\in\Omega$,
\begin{equation}\label{equation section2 7}
\lim_{n\to\infty}\frac{1}{N_n(\omega)}\sum_{k=0}^{n-1}(N_{k+1}(\omega)-N_k(\omega))
 {\bf 1}_{B(\varepsilon)}(\Xi^k(\omega))=\frac{\int_{B(\varepsilon)}N_1(\omega)\,
 dP(\omega)}{\int_{\Omega}N_1(\omega)\,dP(\omega)}=:Q(\vep).
\end{equation}
The fact that $N_1(\omega)$ is $P$-integrable, gives $\lim_{\vep\to 0}Q(\vep)=1$.
Let $Q_L(\vep)$ be as in Lemma~\ref{gooddecay}. Since for $P$-almost all 
$\omega\in\Omega$,
\[
\lim_{L\to\infty}\frac{Q_L(\varepsilon)}{N_L(\omega)}=1-Q(\varepsilon)
\]
and, moreover,
\[
\lim_{L\to\infty}\frac{N_L(\omega)-N_{L-R}(\omega)}{N_L(\omega)}=0
\]
by \eqref{equation section2 2.1}, Lemma~\ref{gooddecay} implies that for 
$P$-almost all $\omega\in\Omega$, 
\[
p^\omega(s)=\lim_{L\to\infty}\frac{1}{N_L(\omega)}\log 
 \sum_{\i_{N_L(\omega)}\in\Sigma^\omega_*}\Phi^s(T^\omega_{\i_{N_L(\omega)}})
\le (1-Q(\varepsilon))\log M\xrightarrow[\varepsilon\to 0]{}0.
\]
From this, we conclude that $s\ge s_0$.
\end{proof}

\begin{remark}\label{finalremark}
As pointed out in \cite{JJKKSS},  the upper bound $1/2$ for 
$\overline\sigma$ is optimal in Theorem~\ref{main theorem}. Moreover, 
it is quite natural to assume that the first neck level is integrable.
Without the integrability condition it is difficult to utilise the shift 
invariance and ergodicity of $P$. If $P$ is not ergodic, one may
use the ergodic decomposition of $P$ and apply Theorem~\ref{main theorem} to 
ergodic components of $P$. However, in this case, the dimension may depend on 
$\omega\in\Omega$. Therefore, the assumptions of Theorem~\ref{main theorem}
may be regarded as optimal ones. 
\end{remark}

\end{document}